\documentclass{scrartcl}

\usepackage[utf8]{inputenc}
\usepackage[T1]{fontenc}
\usepackage{fullpage}
\usepackage{amsmath,amsthm,amssymb,mathtools}

\usepackage{hyperref}
\hypersetup{
  colorlinks=true,
  linkcolor=black,
  citecolor=black,
  urlcolor=blue,
  pdftitle={Possible points of blow-up in chemotaxis systems with spatially heterogeneous logistic source},
  pdfauthor={},
  pdfkeywords={},
  bookmarksopen=true,
}

\usepackage[numbers]{natbib}

\newtheorem{thm}{Theorem}[section]
\newtheorem{lem}[thm]{Lemma}

\theoremstyle{definition} 
\newtheorem{rem}[thm]{Remark}

\usepackage{newunicodechar} 
\newunicodechar{μ}{\mu}
\newunicodechar{ψ}{\psi}
\newunicodechar{φ}{\varphi}
\newunicodechar{ε}{\varepsilon}
\newunicodechar{∞}{\infty}
\newunicodechar{Δ}{\Delta}
\newunicodechar{∇}{\nabla}
\newunicodechar{κ}{\kappa}
\newunicodechar{η}{\eta}
\newunicodechar{ℝ}{\mathbb{R}}
\newunicodechar{α}{\alpha}
\newunicodechar{β}{\beta}
\newunicodechar{ν}{\nu}
\newunicodechar{∂}{\partial}
\newunicodechar{δ}{\delta}
\newunicodechar{Φ}{\Phi}
\newunicodechar{χ}{\chi}
\newunicodechar{ξ}{\xi}

\newcommand{\Om}{\Omega}
\newcommand{\Ombar}{\bar \Omega}
\newcommand{\domain}{G}
\newcommand{\domainbar}{\bar G}
\newcommand{\dOm}{\partial\Om}
\newcommand{\delny}{\partial_{\nu}}
\newcommand{\norm}[2][]{\|#2\|_{#1}}
\newcommand{\Lom}[1]{L^{#1}(\Om)}
\newcommand{\Wom}[2]{W^{#1,#2}(\Om)}
\newcommand{\Tmax}{T_{\mathrm{max}}}
\newcommand{\f}[2]{\frac{#1}{#2}}
\newcommand{\ddt}{\frac{\mathrm{d}}{\mathrm{d}t}}%or whichever formatting is typographically correct ...
\newcommand{\io}{\int_{\Om}}
\newcommand{\nn}{\nonumber}
\newcommand{\kl}[1]{\left(#1\right)}
\newcommand{\set}[1]{\{#1\}}
\newcommand{\matr}[1]{\begin{pmatrix}#1\end{pmatrix}}

\newcommand{\dy}{\,\mathrm{d}y}

\newcommand{\mc}[1]{\mathcal{#1}}

\newcommand{\ur}[1]{\mathrm{#1}}
\newcommand{\ure}{\ur e}

\newcommand{\leb}[2][\Omega]{{L^{#2}(#1)}}
\newcommand{\sob}[3][\Omega]{{W^{#2, #3}(#1)}}
\newcommand{\con}[2][\Ombar]{{C^{#2}(#1)}}

\newcommand{\hp}{\hphantom}
\newcommand{\pe}{\mathrel{\hp{=}}}

\newcommand{\eps}{\varepsilon}

\newcommand{\gt}{>}
\newcommand{\lt}{<}

\newcommand{\tmax}{T_{\max}}

\newcommand{\defs}{\coloneqq}

\newcommand{\R}{\mathbb{R}}

\newcommand{\N}{\mathbb{N}}

\DeclareMathOperator{\supp}{supp}
\DeclareMathOperator{\dist}{dist}

\usepackage{authblk}

\author[1]{Tobias~Black\footnote{e-mail: tblack@math.upb.de}}
\author[2]{Mario~Fuest\footnote{e-mail: fuest@ifam.uni-hannover.de}}
\author[2]{ \\Johannes~Lankeit\footnote{e-mail: lankeit@ifam.uni-hannover.de}}
\author[3]{Masaaki~Mizukami\footnote{e-mail: masaaki.mizukami.math@gmail.com}}

\affil[1]{Institut f\"ur Mathematik, Universit\"at Paderborn, Warburger Str.\ 100, 33098 Paderborn, Germany}
\affil[2]{Leibniz Universität Hannover, Institut für Angewandte Mathematik,\newline Welfengarten 1, 30167 Hannover, Germany}
\affil[3]{Department of Mathematics, Faculty of Education, 
Kyoto University of Education, \newline 1, Fujinomori, Fukakusa, Fushimi-ku, Kyoto 612-8522, Japan} 

\date{}

\title{Possible points of blow-up in chemotaxis systems with spatially heterogeneous logistic source}

\usepackage[dvipsnames]{xcolor}
\DeclareOldFontCommand{\it}{\normalfont\itshape}{\mathit}

\begin{document}
\maketitle  

\renewcommand{\thefootnote}{\fnsymbol{footnote}} 
\footnote[0]
    {2020{\it Mathematics Subject Classification}\/. 
    35B44, 35K55, 92C17.
    }

\KOMAoptions{abstract=true}
\begin{abstract}
\noindent
We discuss the influence of possible spatial inhomogeneities in the coefficients of logistic source terms in  parabolic--elliptic chemotaxis-growth systems of the form
\begin{align*}
 u_t &= \Delta u - \nabla\cdot(u\nabla v) + \kappa(x)u-\mu(x)u^2, \\
 0 &= \Delta v - v + u
\end{align*}
in smoothly bounded domains $\Omega\subset\mathbb{R}^2$. Assuming that the coefficient functions satisfy $\kappa,\mu\in C^0(\overline{\Omega})$ with $\mu\geq0$ we prove that finite-time blow-up of the classical solution can only occur in points where $\mu$ is zero, i.e.\ that the blow-up set $\mathcal{B}$ is contained in
\begin{align*}
\big\{x\in\overline{\Omega}\mid\mu(x)=0\big\}.
\end{align*}
Moreover, we show that whenever $\mu(x_0)>0$ for some $x_0\in\overline{\Omega}$, then one can find an open neighbourhood $U$ of $x_0$ in $\overline{\Omega}$ such that $u$ remains bounded in $U$ throughout evolution.
\end{abstract} \newpage

\section{Introduction}
The interplay between chemotaxis -- motion partially directed in response to concentration differences of a chemical signal -- and population growth can cause quite colourful dynamics \cite{PainterHillenPhysicaD2011,WoodwardEtAlSpatiotemporalPatternsGenerated1995,LankeitChemotaxisCanPrevent2015}. 

In addition, if organisms interact with their surroundings, it is reasonable to assume that -- except in very specific experimental conditions -- the environment is not spatially homogeneous and that, accordingly, terms modelling this interaction should include some space dependence. A question that then becomes immediate for the analysis of such models is that to what extent spatial heterogeneity influences the behaviour of their solutions. 

In this article, we will consider the following chemotaxis-growth system: 
\begin{subequations}\label{system}
\begin{align}
 u_t &= Δu - ∇\cdot(u∇v) + κ(x)u-μ(x)u^2 && \text{in } \Om\times(0,T) \label{system:u}\\
 0 &= Δv - v + u && \text{in } \Om\times(0,T) \label{system:v}\\
 \delny u &= 0 = \delny v && \text{on } \dOm\times(0,T)\label{system:bc}\\
 u(\cdot,0)&=u_0 && \text{in } \Om \label{system:ic} 
\end{align}
\end{subequations}
in a smooth, bounded domain $\Omega\subset ℝ^n$, $n=2$, where the coefficients $\kappa$ and $\mu$ of the growth of logistic type may be spatially inhomogeneous, and will investigate \textit{where} blow-up can ensue. \\

\textbf{Logistic terms in chemotaxis models.}
Compared to the classical Keller--Segel system (whose parabolic--elliptic version is given by \eqref{system} upon setting $κ\equiv 0$, $μ\equiv 0$), which is well-known for its ability to produce solutions blowing up for some and globally existing solutions for other initial data \cite{JL92,Nagai_95,herreroSingularityPatternsChemotaxis1996,NagainSenba_98,Biler_JAMA99,NagaiBlowupNonradialSolutions2001,nagaiGlobalExistenceBlowup2001}, the inclusion of logistic terms (biologically meaningful inter alia when modelling population dynamics \cite{ShigesadaEtAlSpatialSegregationInteracting1979,HillenPainterUserGuidePDE2009}, embryogenesis \cite{PainterEtAlDevelopmentApplicationsModel2000}, tumour invasion processes \cite{ChaplainLolasMathematicalModellingCancer2005} or pattern formation in colonies of \textit{salmonella typhimurium} \cite{WoodwardEtAlSpatiotemporalPatternsGenerated1995}) usually facilitates global existence proofs: If, for example, $κ$ and $μ$ in \eqref{system} are constant and $μ>\f{n-2}n$, then solutions are globally classical and bounded \cite{TelloWinklerChemotaxisSystemLogistic2007}; similarly, in the parabolic--parabolic analogue of \eqref{system}, classical solutions are global if $μ$ is a positive constant and $n=2$ \cite{OsakiEtAlExponentialAttractorChemotaxisgrowth2002} or if $n\ge 3$ and $μ$ is sufficiently large \cite{WinklerBoundednessHigherdimensionalParabolic2010,XiangHowStrongLogistic2018,XiangChemotacticAggregationLogistic2018}. Also existence of weak solutions has been investigated in both parabolic--elliptic (e.g.\ \cite{TelloWinklerChemotaxisSystemLogistic2007,WinklerChemotaxisLogisticSource2008}) and parabolic--parabolic settings (see \cite{LankeitEventualSmoothnessAsymptotics2015,viglialoroVeryWeakGlobal2016}, and \cite{WinklerRoleSuperlinearDamping2019,YanFuestWhenKellerSegel2021,WinklerSolutionsParabolicKeller2020} for even more generalized solutions).
 
The presence of logistic growth and absorption terms of general logistic type, however, does not render the emergence of structure impossible: Not only have solutions been observed to surpass arbitrary population thresholds \cite{WinklerHowFarCan2014,LankeitChemotaxisCanPrevent2015,KangStevensBlowupGlobalSolutions2016,WinklerEmergenceLargePopulation2017}, but even blow-up can occur if the degradation is not too strong: For a simplified variant of \eqref{system} (with \eqref{system:v} reading $0=Δv-m(t)+u$, where $m(t)=\f1{|\Om|}\io u(\cdot,t)$), blow-up has been found for logistic terms of the form $+κu-μu^{α}$ if either $\alpha = 2$ and $\mu \in (0, \frac{n-4}{n})$ (and hence $n \ge 5$) or if $n \ge 3$, $\alpha < \min\{2, \frac{n}{2}\}$ and $\mu > 0$ is arbitrarily large~\cite{FuestApproachingOptimalityBlowup2021} (see also the precedents and related works \cite{WinklerBlowupHigherdimensionalChemotaxis2011, WinklerFinitetimeBlowupLowdimensional2018, BlackEtAlRelaxedParameterConditions2021, TanakaBoundednessFinitetimeBlowup2022, TanakaBlowupQuasilinearParabolic2022}, some of which also do not require the simplification of \eqref{system:v}).

\textbf{Spatial heterogeneity.}
While taking $κ$ and $μ$ to be constant is a useful simplification, in biological reality environmental conditions are rarely perfectly identical everywhere. 
Moreover, taking into account spatial variation in these conditions (as encoded in $x$-dependence of $κ$ and $μ$) can have noticeable influence also on qualitative behaviour of solutions: In \cite{DockeryEtAlEvolutionSlowDispersal1998}, for example, the inclusion of space-dependent reproduction rates in a competition model for two species with different diffusion speed led to extinction of one species, in contrast to the spatially homogeneous version featuring coexistence steady states.

In light of these observations, also chemotaxis-growth systems with non-constant $κ$ and $μ$ have been considered: 
Salako and Shen treated the case that $\Omega = \mathbb{R}^n$, $n\in \mathbb{N}$, and $κ$, $μ$ are positive functions depending on $x\in \mathbb{R}^n$ and $t\in \mathbb{R}$. They showed global existence and uniform persistence, asymptotic spreading of classical solutions \cite{SalakoShenParabolicellipticChemotaxisModel2018b}, established stability of strictly positive entire solutions \cite{SalakoShenParabolicellipticChemotaxisModel2018a} and found conditions for existence and non-existence of transition front solutions \cite{SalakoShenParabolicellipticChemotaxisModel2022}. For bounded domains $\Omega$ and $\kappa$, $\mu$ depending on the space variable only, global solutions have been constructed in \cite{YanFuestWhenKellerSegel2021} (within a generalized solution framework) and \cite{ArumugamEtAlGlobalExistenceSolutions2021} (for a related system), while solutions blowing up in finite time have been found in \cite{FuestFinitetimeBlowupTwodimensional2020, BlackEtAlRelaxedParameterConditions2021} (see also \cite{TanakaBlowupQuasilinearParabolic2022}).

\textbf{The location of blow-up.}
If strong aggregation (in the form of blow-up) has been observed in some model, one of the next questions is \textit{where} it takes place. In the two-dimensional parabolic--parabolic and parabolic--elliptic Keller--Segel model without sources, it has been shown in \cite{SenbaSuzukiChemotacticCollapseParabolicelliptic2001,NagaiSenbaSuzuki} that twice the number of blow-up points in the interior of the domain plus the number of blow-up points on the boundary may not exceed $\frac{\int_\Omega u_0}{4\pi}$. (In the parabolic--parabolic setting, one needs to additionally assume that the blow-up points are isolated, which is the case in radially symmetric settings, where blow-up can only occur at the origin, or when the Lyapunov function is bounded from below; see again \cite{NagaiSenbaSuzuki}). Thus, for initial data $u_0$ with $\int_\Omega u_0 \in [4\pi, 8\pi)$, blow-up can only occur at the boundary.
At least in the parabolic--elliptic model, such boundary blow-up has indeed been detected \cite[Theorem~3.2]{NagaiBlowupNonradialSolutions2001}.

Moreover, again for the two-dimensional parabolic--elliptic system, it is known that single-point blow-up may occur at any point in $\Ombar$ \cite{NagaiBlowupNonradialSolutions2001}.
Regarding the fully parabolic setting, \cite{HerreroVelazquezBlowupMechanismChemotaxis1997} shows the existence of a solution only blowing up at the center of the disc.
In fact, many available techniques for detecting blow-up require radially symmetry so that one often obtains blow up at the origin as a by-product.
Beyond that, blow-up points can be characterised by their temporal asymptotics \cite{MizoguchiSouplet}.
Apart from the location of blow-up points, many more properties of such nonglobal solutions have been studied; we refer to Section~2.2 of the survey \cite{LankeitWinklerFacingLowRegularity2019} and the references therein for further discussion.

All of these works do not deal with logistic terms. This is natural: In some cases, the methods of the proofs would not apply, 
for example, those in \cite{NagaiSenbaSuzuki} rest on a Lyapunov functional for the Keller--Segel system, which no longer is a Lyapunov functional for \eqref{system} if $κ$ and $μ$ differ from zero. Even more importantly, for planar domains any positive constant $μ$ suffices to guarantee boundedness of classical solutions \cite{TelloWinklerChemotaxisSystemLogistic2007}, as discussed above, or to ensure that solutions emerging from $L^1$ initial data immediately become regular \cite{LankeitImmediateSmoothingGlobal2020}. 

On the other hand, these latter results already indicate that also for nonconstant $μ$ one might rather expect blow-up in places where $μ$ vanishes. We show that this speculation indeed holds true: 

\begin{thm}\label{th:blowup_points}
  Let $\Omega \subset \R^2$ be as smooth, bounded domain, $\kappa, \mu \in \con0$ with $\mu \ge 0$ and $0 \le u_0 \in \con0$.
  Suppose $(u, v)$ solves \eqref{system} classically and blows up at $\tmax \lt \infty$.
  Then the blow-up set
  \begin{align*}
    \mc B \defs \left\{x \in \Ombar:
      \exists (x_j)_{j \in \N} \subset \Ombar,\, (t_j)_{j \in \N} \subset (0, \tmax) \mid
      x_j \to x, t_j \to \tmax, u(x_j, t_j) \to \infty
    \right\}
  \end{align*}
  is contained in $\big\{ x \in \Ombar \mid \mu(x) = 0\big\}$.
\end{thm}

We prove this theorem by showing the following: 

\begin{thm}\label{th:main}
  Let $\Omega \subset \R^2$ be as smooth, bounded domain, $\kappa, \mu \in \con0$ with $\mu \ge 0$ and $0 \le u_0 \in \con0$.
  Suppose $(u, v)$ is a classical solution of \eqref{system} in $\Ombar\times [0, \tmax)$, where $\tmax \in (0, \infty]$.

  Then the following holds:
  Let $x_0 \in \Ombar$ with $\mu(x_0) > 0$ and $T \in (0, \tmax] \cap (0, \infty)$.
  Then there exist an open neighbourhood $U$ of $x_0$ in $\Ombar$ and $C \gt 0$ such that
  \begin{align*}
    \|u(\cdot, t)\|_{L^\infty(U)} \lt C 
    \qquad \text{for all $t \in (0, T)$}.
  \end{align*}
\end{thm}

\begin{rem}
  We emphasize that Theorems~\ref{th:blowup_points} and \ref{th:main} do not rely on any symmetry of the domain nor the data,
  keeping this result apart from the above-mentioned works, where the only possible location of blow-up is known to be the origin, as by-product of the methods employed and the radial symmetry assumption thereby required.
\end{rem}

\textbf{Plan of the paper.}
After establishing boundedness properties of low level but global in space in Section~\ref{sec3}, we will continue in Section~\ref{sec2} with the construction of smooth cut-off functions, which satisfy certain appropriate estimates crucial for the localized estimations performed in Sections~\ref{sec4} and \ref{sec5}. The essential steps of our analysis are undertaken in Section~\ref{sec4}. There, the properties of the cut-off functions $\varphi$, which are supported in the positivity set of $\mu$, enable us to derive an ODE for a functional of the form $y(t):=\io \varphi u^p(\cdot,t)$ for certain beneficial values of $p>1$, which thereby provides localized $L^p$-boundedness information for $u$ on finite time-intervals in light of bounds for $\io\varphi u^p$ (Lemma~\ref{lm:u_lp}). In a second step this information can be refined to a bound for $\io|\nabla(\varphi v)|^q$ for some $q>2$ (Lemma~\ref{lm:varphi_v_w1q}). In neighbourhoods $V$ of points where $\mu$ is positive we will then be able to establish the boundedness of $\norm[L^{q}(V)]{\nabla v}$ on finite time intervals (Lemma~\ref{lm:local_bound_1}). As $q$ is larger than the space dimension, this information is sufficient for the semigroup techniques employed in Section~\ref{sec5}, where we prove that $u$ is locally bounded in $\Lom{\infty}$ for all finite time intervals (Lemma~\ref{lm:l_infty_bdd} and Lemma~\ref{lm:local_bound_2}).

\section{Spatially global bounds}\label{sec3}
In the sequel, we fix a smooth, bounded domain $\Omega \subset \R^2$, $\kappa, \mu, u_0 \in \con0$ with $\mu \ge 0$ and $u_0 \ge 0$ and a classical solution $(u, v)$ of \eqref{system} with maximal existence time $\tmax \in (0, \infty]$.

Let us start by gathering some time-global boundedness information for the solution components.

\begin{lem}\label{lm:u_l1}
For every finite $T\le \Tmax$ there is $C_T>0$ such that 
$\norm[\Lom1]{u(\cdot,t)} \le C_T$ for every $t\in[0,T)$. 
\end{lem}
\begin{proof}
In view of \eqref{system:u} and \eqref{system:bc}, integrating by parts and utilizing the nonnegativity of $\mu$ reveals that $y(t):=\io u(\cdot,t)$, $t\in(0,\Tmax)$, satisfies
\begin{align*}
y'(t)=\io \kappa u(\cdot,t)-\io\mu u^2(\cdot,t)\leq\norm[\Lom \infty]{\kappa} y(t)\qquad\text{for all }t\in(0,\Tmax).
\end{align*}
Hence, setting $\overline{y}(t):=\norm[\Lom1]{u_0} \ure^{t\norm[\Lom \infty]{\kappa}}$, $t\in(0,\Tmax)$, we conclude from an ODE comparison argument that $y(t)\le C_T:= \norm[\Lom1]{u_0} \ure^{T\norm[\Lom \infty]{\kappa}}$ for all $t\in[0,T)$.
% \begin{align*}
% \io u(\cdot,t)\leq \norm[\Lom1]{u_0} \ure^{t\norm[\Lom \infty]{\kappa}}\quad\text{for all }t\in[0,\Tmax).
% \end{align*}
% Then, given any finite $T\le \Tmax$ the claim follows upon letting $C_T:= \norm[\Lom1]{u_0} \ure^{T\norm[\Lom \infty]{\kappa}}$.
\end{proof}

\begin{lem}\label{lm:v_lq}
For every $p\in[1,2)$ and finite $T\leq \Tmax$ there is $C_1>0$ such that 
\[
\norm[W^{1,p}(\Omega)]{v(\cdot,t)}<C_1\qquad\text{for all }t\in(0,T).
\]
Moreover, for every $q<\infty$ and finite $T\leq \Tmax$ there is $C_2>0$ such that
\[
\norm[\Lom q]{v(\cdot,t)}< C_2\qquad\text{for all }t\in(0,T).
\]
\end{lem}

\begin{proof}
We fix $p\in[1,2)$ and find that according to a result from elliptic regularity theory (cf.\ \cite[Lemma~23]{BrezisStrauss73})  there is $C_0=C_0(p)>0$ such that
\begin{align*}
\norm[W^{1,p}(\Omega)]{v}\leq C_0\norm[\Lom 1]{(\Delta-1)v}=C_0\norm[\Lom 1]{u}\qquad\text{in }(0,\Tmax),
\end{align*}
which in light of Lemma~\ref{lm:u_l1} shows that for any finite $T\leq \Tmax$ there is $C_1=C_1(p,T)>0$ such that
\begin{align*}
\norm[W^{1,p}(\Omega)]{v(\cdot,t)}\leq C_1\qquad\text{for all }t\in(0,T).
\end{align*}
The second assertion then follows from the first and the Sobolev embedding theorem in light of the fact that for any $q<\infty$ we can pick $p\in[1,2)\cap[\f{2q}{q+2},2)$.
\end{proof}

\section{Suitable cut-off functions}\label{sec2}
In order to derive localized bounds which go beyond those implied by the global estimates obtained in the previous section, we will make use of certain (spatial) cut-off functions.
The general idea is to fix $x_0 \in \Ombar$ with $\mu(x_0) > 0$ and then choose a cut-off function on whose support $\mu$ is still positive. 
Apart from usual properties of cut-off functions, we inter alia also need vanishing normal derivatives. Since these functions are of crucial importance for the proof, we go into some detail and now construct cut-off functions suitable for our purpose.

Moreover, we choose to formulate the statements in this section for general (bounded and two-dimensional) $C^l$-domains $G$, not only for the already fixed smooth domain $\Omega$. However, in later sections, we will always apply these results to $G = \Omega$.
\begin{lem}\label{domaintrafo}
Let $\domain\subset ℝ^2$ be a $C^l$-domain, $l\ge1$, and $(x_0,y_0)\in \domainbar$. Then there are open sets $B,W\subset ℝ^2$ with $(x_0,y_0)\in W$ and a $C^l$-diffeomorphism $Φ\colon W\to B$ such that 
\begin{subequations}\label{diffeo-1}
\begin{align}
 Φ^{-1}(B\cap (ℝ\times(0,∞))) &= W\cap \domain, \label{diffeo-1:inner}\\
 Φ^{-1}(B\cap (ℝ\times\set{0})) &= W\cap \partial\domain, \label{diffeo-1:bdr}\\
 Φ^{-1}(B\cap (ℝ\times(-∞,0))) &= W\setminus \domainbar 
\end{align}
\end{subequations}
and 
\begin{equation}\label{diffeo-normal}
 ∂_{ν} Φ_1 (x,y)=0,\qquad ∂_{ν} Φ_2(x,y)<0 \qquad \text{for every } (x,y)\in∂\domain\cap W.
\end{equation}
\end{lem}
\begin{proof}
 If  $(x_0,y_0)\in \domain$, choosing $W$ sufficiently small ensures that \eqref{diffeo-normal} is an empty condition and of \eqref{diffeo-1} only the first part \eqref{diffeo-1:inner} is relevant and can easily be obtained by choosing a suitable translation as $Φ$. We treat the case $(x_0,y_0)\in ∂\domain$ in more detail:  After rotation if necessary, we may assume that on an open neighbourhood $W_1$ of $(x_0,y_0)$ we have $(x,y)\in \domain$ iff $y>f(x)$ for some $C^l$-function $f$ satisfying $f(x_0)=y_0$ and $f'(x_0)>0$.
 We let $v\colon(x_0-δ,x_0+δ)\to ℝ$ be the solution of 
 \[
  v'(x) = \f1{f'(x)} v(x),\quad x\in (x_0-δ,x_0+δ),\qquad v(x_0)=1,
 \]
 where $δ>0$ is chosen sufficiently small that this problem has a unique solution and, moreover, $f'(x)>0$ for all $x\in (x_0-δ,x_0+δ)$. For $(x,y)\in W_2:=W_1\cap ((x_0-δ,x_0+δ)\timesℝ)$ we then define 
 \[
  Φ(x,y)=\matr{v(x)\ure^y-\ure^{y_0}\\y-f(x)}
 \]
 and note that due to 
 \begin{equation}\label{trafo-derivative}
  DΦ(x,y)=\matr{v'(x)\ure^y&v(x)\ure^y\\-f'(x)&1}=\matr{\f{v(x)}{f'(x)}\ure^y&v(x)\ure^y\\-f'(x)&1}
 \end{equation}
we have $\det DΦ(x_0,y_0)=\f{1}{f'(x_0)}\ure^{y_0}+f'(x_0)\ure^{y_0}>0$ and hence there is an open neighbourhood $W\subset W_2$ of $(x_0,y_0)$ such that the restriction $Φ\colon W\to Φ(W)=:B$ is a $C^l$-diffeomorphism.

We observe that for $(x,y)\in W$ we have 
\[
(x,y)\in \domain\iff y>f(x)\iff Φ_2(x,y)>0\iff Φ(x,y)\in B\cap (ℝ\times (0,∞))
\]
and similarly for the other parts of \eqref{diffeo-1}. Moreover, on account of \eqref{trafo-derivative}, and the unit normal at any $(x,y)\in W_1\cap ∂\domain$ being given by 
\[
 ν(x,y) = \f1{\sqrt{1+(f'(x))^2}} \matr{f'(x)\\-1},
\]
we have 
\begin{align*}
 ∂_{ν} Φ(x,y) &= DΦ(x,y) ν(x,y) \\ &= \f1{\sqrt{1+(f'(x))^2}} \matr{\f{v(x)}{f'(x)}\ure^y&v(x)\ure^y\\-f'(x)&1}\matr{f'(x)\\-1} = \matr{0\\-\sqrt{1+(f'(x))^2}}
\end{align*}
for all $(x, y) \in W$, which implies \eqref{diffeo-normal}.
\end{proof}

Utilizing the domain transformation provided by the lemma above, we can now construct suitable spatial cut-offs by the following two-step procedure.

\begin{lem}\label{lm:nicefunction}
 Let $\domain\subset ℝ^2$ be a bounded $C^l$-domain, $l\ge 1$, and $x_0\in \domainbar$. Then there is a set $W\subset \domainbar$ with $x_0\in W$ and $W$ relatively open in $\domainbar$ such that whenever $K\subset W$ is compact and $V$ is an open neighbourhood of $K$ in $W$ with $\bar V\subset W$, there is a function $φ\in C^l(\domainbar)$ such that 
 \[
  0\le φ\le 1 \text{ in } \domainbar, \quad φ=0 \text{ in } \domainbar\setminus V, \quad φ=1 \text{ in } K \quad \text{and} \quad ∂_{ν} φ=0 \text{ on } ∂\domain.
 \]
\end{lem}
\begin{proof}
 We choose open sets $W, B$ and a $C^l$-diffeomorphism $Φ$ as in Lemma~\ref{domaintrafo}. Given $K$ and $V$ as in the statement of this lemma, we let 
 \begin{align*}
  V'&= \set{\,x\in ℝ^2\mid (x_1, x_2)\in Φ(V) \text{ or } (x_1, -x_2) \in \Phi(V)\,}
 \intertext{and}
  K'&= \set{\,x\in ℝ^2\mid (x_1, x_2)\in Φ(K) \text{ or } (x_1, -x_2) \in \Phi(K)\,}.
 \end{align*}
Then $K'\subset V'$ and $K'$ is compact.
Moreover, $V'$ is open in $ℝ^2$ since for each $z \in V'$ there is an open neighbourhood of $z$ contained in $V'$, which can be shown by considering each of the cases $z_2 > 0$, $z_2 < 0$ and $z_2 = 0$ separately.
%Let $(z_1,z_2)\in V'$. Show: There is an open neighbourhood of $z$ which is subset of $V'$. Case i): $z_2>0$: Then $Φ^{-1}(z)\in V\cap \domain$ and there is an open set $U\subset V\cap \domain$ with $Φ^{-1}(z)\in U$; $Φ(U)\subset V'$ and $Φ(U)$ is an open neighbourhood of $z$. Case ii): $z_2<0$: Find a neighbourhood $W$ of $(z_1,-z_2)$ as in Case i) and let $W'=\set{(a_1,-a_2)\mid a\in W}$. Case iii): $z_2=0$. Then $z'=Φ^{-1}(z)\in V\cap ∂\domain$. Let $\tilde V\subset ℝ^2$ be open such that $\tilde V\cap \domainbar=V$. There is an open ball $B_{δ}(z')\subset \tilde V$ and $Φ(B_{δ}(z'))$ contains an open ball $B_{ε}(z)$. Claim: $B_{ε}(z)\subset V'$. Let $a\in B_{ε}(z)$. If $a_2\ge 0$, then $Φ^{-1}(a)\in Φ(B_{ε}(z))\cap \domainbar \subset B_{δ}(z')\cap \domainbar\subset \tilde V\cap \domainbar=V$, therefore $a\in V'$.
%}

If $K'=\emptyset$, we let $\tilde{φ}=0$, otherwise we introduce \(δ=\dist(K',ℝ^2\setminus V')>0\), let $χ$ be the characteristic function of $\set{y\inℝ^2\mid \dist(y,K')<\f{δ}3}$, let $ξ\in C^\infty([0,∞))$ be a decreasing function with $ξ(s)=1$ if $s<\f{δ}3$, $ξ(s)=0$ if $s>\f{2δ}3$ and set 
\[
 \tilde{φ}(x) \defs c_0\int_{ℝ^2} χ(y) ξ(|x-y|) \dy \quad \text{for $x \in \R^2$, where} \quad c_0 \defs \left(\int_{ℝ^2} ξ(|y|) \dy \right)^{-1}.
\]
Then $\tilde{φ}(x)=0$ if $x\in ℝ^2\setminus V'$, $\tilde{φ}(x)=1$ if $x\in K'$, $\tilde{φ}(x)\in [0,1]$ for every $x\in ℝ^2$ and $\tilde{φ}\in C^{∞}(ℝ^2)$. Moreover, for any $z=(z_1,z_2)\in ℝ^2$ denoting $(z_1,-z_2)$ by $\hat z$, we have that
\begin{align*}
 \tilde{φ}(\hat x)&=c_0 \int_{ℝ^2} χ(y) ξ(|\hat x-y|)\dy = c_0 \int_{ℝ^2} χ(y) ξ(|x-\hat y|)\dy \\
 &= c_0\int_{ℝ^2} χ(\hat y)ξ(|x-\hat y|)\dy=c_0\int_{ℝ^2}χ(y)ξ(|x-y|)\dy=\tilde{φ}(x)
\end{align*}
for all $x \in \R^2$,
where the third equality relies on the symmetry of $K'$. Therefore, $\tilde{φ}_{x_2}(x_1, x_2)=0$, whenever $x\in W$ and $x_2=0$. Finally, we set
\[
 φ(x)=\tilde{φ}(Φ(x)),\qquad x\in W, 
\]
and observe that if $x\in ∂\domain$, then by the choice of $Φ$, \eqref{diffeo-1:bdr} and \eqref{diffeo-normal},
\[
   ∂_{ν} φ(x)
 = \tilde{φ}_{x_1} (Φ(x_1, x_2)) \underbrace{∂_{ν} Φ_1(x_1, x_2)}_{=0} + \underbrace{\tilde{φ}_{x_2} (Φ(x_1, x_2))}_{=0} ∂_{ν} Φ_2(x_1, x_2) = 0.\qedhere
\]
\end{proof}

\begin{lem}\label{lm:cutoff}
Let $\domain \subset \R^2$ be a bounded $C^2$-domain and $x_0\in \domainbar$. Then there is a relatively open set $W\subset \domainbar$ containing $x_0$ such that whenever $K \subset W$ is compact, $V$ is an open neighbourhood of $K$ in $W$ and $\eta \in (0, \frac12)$, 
  then there exist $\varphi \in C^2(\domainbar)$ and $C_\varphi \gt 0$ such that
  \begin{align}\label{eq:cutoff:varphi}
    0\le φ\le 1 \text{ in } \domainbar, \quad φ=0 \text{ in } \domainbar\setminus V, \quad φ=1 \text{ in } K \quad \text{and} \quad ∂_{ν} φ=0 \text{ on } ∂\domain
  \end{align}
  and
  \begin{align}\label{eq:cutoff:grad_est}
    |\nabla \varphi| \le C_\varphi \varphi^{1-\eta}
    \quad \text{and} \quad
    |\Delta \varphi| \le C_\varphi \varphi^{1-2\eta}
    \qquad \text{in $\domainbar$}.
  \end{align}
\end{lem}
\begin{proof}
  With $W$ as in Lemma~\ref{lm:nicefunction}, we let $\tilde \varphi \in C^{2}(\domainbar)$ be a function
  with $\partial_\nu \tilde \varphi = 0$ on $\partial \domain$,
  $0 \le \tilde \varphi \le 1$,
  $\tilde \varphi = 0$ in $\domainbar \setminus V$ and
  $\tilde \varphi = 1$ in $K$ as provided by Lemma~\ref{lm:nicefunction}. Since
  \begin{align*}
        |\nabla \tilde \varphi^\frac1\eta| 
    =   \frac1\eta \tilde \varphi^{\frac1\eta-1} |\nabla \tilde \varphi|
    \le \frac{\|\tilde \varphi\|_{C^1(\domainbar)}}{\eta} \left(\tilde \varphi^\frac1\eta\right)^{1-\eta}
  \end{align*}
  and
  \begin{align*}
        |\Delta \tilde \varphi^\frac1\eta| 
    =   \frac1\eta \nabla \cdot (\varphi^{\frac1\eta-1} \nabla \varphi)
    =   \frac{1-\eta}{\eta^2} \varphi^{\frac1\eta-2} |\nabla \varphi|^2
        + \frac{\eta}{\eta^2} \varphi^{\frac1\eta-1} |\Delta \varphi|
    \le \frac{1}{\eta^2} \|\tilde \varphi\|_{C^2(\domainbar)}^2  \left(\tilde \varphi^\frac1\eta\right)^{1-2\eta}
  \end{align*}
  in $\domainbar$,
  $\varphi \defs \tilde \varphi^\frac1\eta$ additionally fulfils \eqref{eq:cutoff:grad_est} for a certain $C_\varphi \gt 0$.
\end{proof}

\begin{rem}\label{rm:eta_optimal}
  Positivity of $\eta$ in Lemma~\ref{lm:cutoff} is a necessary condition in the sense that there are no nontrivial $0 \le \varphi \in C_c^{2}(\domain)$ fulfilling \eqref{eq:cutoff:grad_est} for $\eta = 0$.
  Indeed, if $\varphi(x_0) = 0$ for some $x_0 \in \domainbar$ and $|\nabla \varphi \cdot \xi| \le C_\varphi \varphi$ in $\domainbar$ for $\xi \in \partial B_1(0)$ and some $C_\varphi > 0$,
  then $\varphi(x_0 + \xi s) = 0$ for all $s \in \R$ with $x_0 + \xi s \in \domainbar$ by the comparison principle for ODEs.
\end{rem}

\section{\texorpdfstring{A localized $W^{1, q}$ bound for some $q > 1$ for $v$}{A localized W1q bound for some q>1 for v}}\label{sec4}
We now aim to obtain localized bounds which are not already implied by the spatially global bounds obtained in Section~\ref{sec3}.
To that end, a natural idea is to follow the reasoning in \cite{TelloWinklerChemotaxisSystemLogistic2007},
where global existence of classical solutions of \eqref{system} has been shown for constant $\mu > 0$,
and to replace $u$ and $v$ by $\varphi u$ and $\varphi v$ for certain cut-off functions $\varphi$ on whose support $\mu$ is uniformly positive, whenever appropriate.
In addition to \cite{TelloWinklerChemotaxisSystemLogistic2007}, however, we then need to also account for terms involving derivatives of $\varphi$.
As it turns out, this would be entirely unproblematic \emph{if} we could choose cut-off functions as in Lemma~\ref{lm:cutoff} which fulfil \eqref{eq:cutoff:grad_est} for $\eta = 0$.
Unfortunately, as already observed in Remark~\ref{rm:eta_optimal}, this is impossible.
However, as we may at least choose $\eta > 0$ arbitrarily close to $0$ in Lemma~\ref{lm:cutoff},
one might still hope that these terms can be favourably estimated, although perhaps in a more delicate way.

This is indeed the case.
Thus, as a first step towards deriving localized $L^\infty$ bounds,
we now bound $\io \varphi u^p$ locally uniformly in time for certain $p>1$.
\begin{lem}\label{lm:u_lp}
  Let $\mu_0 \gt 0$, $p \in (1, \frac{1}{(1-\mu_0)_+})$ and $T \in (0, \tmax] \cap (0, \infty)$.
  Suppose $\varphi \in C^{2}(\Ombar)\setminus\set{0}$ with $\partial_\nu \varphi = 0$ on $\partial \Omega$,
  $\supp \varphi \subset \mu^{-1}((\mu_0, \infty))$
  and $0 \le \varphi \le 1$ in $\Ombar$
  fulfils \eqref{eq:cutoff:grad_est} for some $C_\varphi \gt 0$, $\eta \defs \frac1{2(p+1)}$ and $G \defs \Omega$.
  Then there exists $C\gt 0$ such that
  \begin{align}\label{eq:u_lp:statement}
    \io \varphi u^p(\cdot, t) \le C
    \qquad \text{for all $t \in (0, T)$}.
  \end{align}
\end{lem}
\begin{proof}
Given $T\in(0,\infty)$ such that $T\leq \Tmax$, we test \eqref{system:u} against $\varphi u^{p-1}$ and see from \eqref{system:bc}, integration by parts and appropriate reexpression of some integrands that
 \begin{align}\label{b1}
  \f1p \ddt \io φu^p &=\io \varphi u^{p-1}\Delta u-\io\varphi u^{p-1}\nabla u\cdot\nabla v-\io\varphi u^p\Delta v+\io\kappa\varphi u^p-\io\mu\varphi u^{p+1}\nn\\
  &=-(p-1)\io \varphi u^{p-2}|\nabla u|^2-\io u^{p-1}\nabla\varphi\cdot\nabla u-\frac{1}{p}\io\varphi\nabla u^p\cdot \nabla v-\io\varphi u^p\Delta v\nn\\
  &\pe+\io\kappa\varphi u^p-\io\mu\varphi u^{p+1}\nn\\
  &\le-(p-1)\io \varphi u^{p-2}|\nabla u|^2-\frac{1}{p}\io\nabla u^p\cdot\nabla\varphi+\frac{1}{p}\io\varphi u^p\Delta v \nn\\
  &\pe+\frac{1}{p}\io u^p\nabla\varphi\cdot\nabla v-\io \varphi u^p\Delta v+\io\kappa\varphi u^p-\mu_0\io\varphi u^{p+1}\qquad\text{in }(0,T),
 \end{align}
 where in the last step we have made use of the fact that $\mu \ge \mu_0$ in $\supp \varphi$.
 Herein, we find from \eqref{system:v} that
 \begin{align*}
 \io\varphi u^p\Delta v= \io\varphi u^p(v-u)= \io \varphi u^p v-\io\varphi u^{p+1}\qquad\text{in }(0,T),
 \end{align*}
 and, since $\partial_\nu\varphi=0$ on $\partial\Omega$, another integration by parts moreover entails that
 \begin{align*}
 -\frac{1}{p}\io\nabla u^p\cdot\nabla\varphi=\frac{1}{p}\io u^p\Delta\varphi\qquad\text{in }(0,T).
 \end{align*}
 Plugging these back into \eqref{b1} and neglecting the nonpositive term $-\frac{p-1}{p}\io\varphi u^p v$ then yields
 \begin{align}\label{firstestimateddtup}
 \frac{1}{p} \ddt \io\varphi u^p&\leq -(p-1)\io \varphi u^{p-2}|\nabla u|^2+\frac{1}{p}\io u^p\Delta\varphi+\frac{p-1}{p}\io\varphi u^{p+1}\nn\\
 &\pe+\frac{1}{p}\io u^p\nabla\varphi\cdot\nabla v+\io\kappa\varphi u^p-\mu_0 \io \varphi u^{p+1}\qquad\text{in }(0,T).
 \end{align}

In the term $\f1p\io u^p∇φ\cdot∇v$, we again integrate by parts and make use of $\partial_\nu \varphi = 0$ on $\partial \Omega$ to obtain that
 \begin{align}\label{thebadintegral}
 \f1p \io u^p∇φ\cdot∇v &= -\io u^{p-1}v∇u\cdot∇φ - \f1p\io u^pvΔφ \qquad\text{in }(0,T).
 \end{align}
 Now, since the assumed condition $p<\frac{1}{(1-\mu_0)_+}$ entails that $\mu_0-\frac{p-1}{p}>0$ we can pick $\eps=\eps(\mu_0,p)>0$ satisfying
 \begin{align}\label{choiceofeps}
 \eps<p-1\quad\text{and}\quad 3\eps<\mu_0-\frac{p-1}{p}.
 \end{align}
 By virtue of Young's inequality, we moreover find $C_1=C_1(\eps)>0$ and $C_2=C_2(\eps)>0$ such that
 \begin{align}\label{badintegral-firstpart}
 &\pe-\io u^{p-1}v∇u\cdot∇φ \nn \\
 &\le ε\!\io φu^{p-2}|∇u|^2 + C_1\!\io \f{|∇φ|^2}{φ} v^2 u^p \nn\\
 &=ε\!\io φu^{p-2}|∇u|^2 + C_1\!\io \f{|∇φ|^2}{φ} v^2 φ^{\f{p}{p+1}}u^pφ^{-\f{p}{p+1}}\nn\\
 &\le ε\!\io φu^{p-2}|∇u|^2 + ε\!\io φu^{p+1} + C_2\! \io φ^{-p}\f{|∇φ|^{2(p+1)}}{φ^{p+1}} v^{2(p+1)}\nn\\
 &= ε\!\io φu^{p-2}|∇u|^2 + ε\!\io φu^{p+1} + C_2\! \io \f{|∇φ|^{2(p+1)}}{φ^{2p+1}} v^{2(p+1)}\qquad\text{in }(0,T).
\end{align}
Similarly, Young's inequality provides $C_3=C_3(\eps) > 0$ such that
 \begin{equation}\label{badintegral-secondpart}
- \f1p \io u^pvΔφ \le ε\!\io φu^{p+1} + C_3\! \io  \frac{|Δφ|^{p+1}}{\varphi^p} v^{p+1} \qquad\text{in }(0,T).
 \end{equation}

Since $\eta = \frac1{2(p+1)}$, and since \eqref{eq:cutoff:grad_est} thus entails
  \begin{align}%\label{eq:testfuncestimates-1}
        \varphi^{-2p-1} |\nabla \varphi|^{2(p+1)} 
    &\le C_\varphi^{2p+2} \varphi^{1-2(p+1)\eta}
    =   C_\varphi^{2p+2} \nn
  \intertext{and}\label{eq:testfuncestimates-2}
        \varphi^{-p} |\Delta \varphi|^{p+1}
    &\le C_\varphi^{p+1} \varphi^{1-2(p+1)\eta}
    =   C_\varphi^{p+1}
  \end{align}
  in $(0, T)$,
  \eqref{badintegral-firstpart} and \eqref{badintegral-secondpart} turn  \eqref{thebadintegral} into 
  \begin{align}\label{badintegral-thirdpart}
   \f1p \io u^p∇φ\cdot∇v \le ε\io φu^{p-2}|∇u|^2 + 2ε\io φu^{p+1} + C_4 \io v^{2(p+1)} + C_4 \io v^{p+1}
  \end{align}
 on $(0,T)$, with $C_4=C_4(\varepsilon,p,\varphi,\eta):=\max\big\{C_2 C_\varphi^{2p+2},C_3 C_\varphi^{p+1}\big\}>0$. Likewise, we conclude from Young's inequality, \eqref{eq:testfuncestimates-2} and $p>1$ that there is $C_5=C_5(\eps,p,\varphi,\eta)>0$ satisfying
 \begin{align}\label{other-badintegral}
  \f1p\io u^pΔφ \le ε{C_\varphi^{-\frac {p+1}p}}
  \io φ u^{p+1} φ^{-\f{p}{p}}|Δφ|^{\f{p+1}p} + C_5 \le ε\io φ u^{p+1} +C_5\qquad\text{in }(0,T). 
 \end{align}

Therefore, inserting \eqref{badintegral-thirdpart} and \eqref{other-badintegral} into \eqref{firstestimateddtup} and recalling \eqref{choiceofeps}, which entails $p-1-\eps > 0$ and $\mu_0 - \frac{p-1}{p} - 3\eps > 0$, yields
\begin{align*}%\label{secondtestimateddtup}
  \f1p \ddt \io φu^p 
  &\le -\left( p-1- \varepsilon \right) \io φ u^{p-2} |∇u|^2  + C_4 \io v^{2(p+1)}+C_4 \io v^{p+1} \nn \\ &\pe+ \io κφu^p -  \kl{\mu_0-\f{p-1}p -3ε} \io φu^{p+1}+C_5 \nn \\
  &\le C_4 \io v^{2(p+1)}+C_4 \io v^{p+1} + \|\kappa\|_{\leb\infty} \io φu^p +C_5 \qquad\text{in }(0,T).
\end{align*}
Due to the bounds on $v$ and $κ$ asserted in Lemma~\ref{lm:v_lq} and by the inclusion $\kappa \in \con0$, respectively, the (time-local) estimate \eqref{eq:u_lp:statement} then follows by an ODE comparison argument for a suitably chosen $C > 0$.
\end{proof}

In combination, Lemma~\ref{lm:cutoff} and Lemma~\ref{lm:u_lp} show that if $x_0 \in \Ombar$ is such that $\mu(x_0) > 0$,
then we can find an open neighbourhood $V$ of $x_0$ in $\Ombar$ such that $\int_V u^p$ remains bounded (locally in time) throughout evolution.
Since $\Omega$ is assumed to be a two-dimensional domain,
one generally expects that this bound can be turned into fineteness of $\|u\|_{L^\infty(V' \times (0, T))}$ for some open $V' \subset V$ and all finite $T \in (0, \tmax]$.

There are multiple reasons for this expectation:
On the one hand, in the case of $V=\Omega$ one obtains the desired estimate even for $V'=\Omega$ and $\kappa = \mu \equiv 0$
(see for instance \cite[Lemma 3.2]{BellomoEtAlMathematicalTheoryKeller2015}). 
On the other hand, for the Cauchy problem in $\R^n$ and $\kappa = \mu \equiv 0$ but general open $V \subset ℝ^n$, the $L^\infty$ bound follows from the $\eps$-regularity theorem proved in \cite{SugiyamaVarepsilonregularityTheorem2010}.

None of these results is applicable in our setting, however.
Instead, we next make use of elliptic regularity theory (and Lemma~\ref{lm:u_lp}) to obtain certain gradient bounds for the second solution component.
\begin{lem}\label{lm:varphi_v_w1q}
 Assume $\mu_0 \gt 0$, $p \in (1, \frac{1}{(1-\mu_0)_+})\cap(1,2)$, let $\varphi$ be as in Lemma~\ref{lm:u_lp}, $T \in (0, \tmax] \cap (0, \infty)$
 and let $q=\f{2p}{2-p}>2$. Then there is $C \gt 0$ such that
 \begin{align}\label{eq:varphi_v_w1q:est}
   \io |\nabla (\varphi v(\cdot, t))|^q \le C
   \qquad \text{for all $t \in (0, T)$}.
 \end{align}
\end{lem}
\begin{proof}
Straightforward computations drawing on \eqref{system:v} show that
\begin{align*}
-\Delta(\varphi v)+\varphi v=-v\Delta\varphi+\varphi u-2\nabla\varphi\cdot\nabla v\qquad\text{in }\Omega\times(0,\Tmax).
\end{align*}
Now, given $p\in(1, \frac{1}{(1-\mu_0)_+})\cap(1,2)$ we find that according to elliptic regularity theory (e.g.\ \cite[Theorem I.19.1]{friedman2011partial}) employed to the uniformly elliptic operator given by $-(\Delta-1)$, there is $C_1=C_1(p)>0$ such that
\begin{align*}
\norm[\Wom{2}{p}]{\varphi v(\cdot,t)}&\leq C_1\norm[\Lom{p}]{-v(\cdot,t)\Delta\varphi+\varphi u(\cdot,t)-2\nabla\varphi\cdot\nabla v(\cdot,t)}+C_1\norm[\Lom{p}]{\varphi v(\cdot,t)}
\end{align*}
for $t\in(0,\Tmax)$,
so that $\varphi\in C^2(\bar{\Omega})$ entails that there is some $C_2=C_2(p,\varphi)>0$ satisfying
\begin{align*}
\norm[\Wom{2}{p}]{\varphi v}\leq C_2\norm[\Lom{p}]{v}+C_2\norm[\Lom{p}]{\nabla v}+C_2\norm[\Lom{p}]{u}\qquad\text{in }(0,\Tmax).
\end{align*}
Here, noticing that the restrictions $p\in(1,2)$ and $p\in(1, \frac{1}{(1-\mu_0)_+})$ make both Lemma~\ref{lm:v_lq} and Lemma~\ref{lm:u_lp} applicable, we conclude that for any finite $T\leq \Tmax$ there is $C_3=C_3(\mu_0,p,\varphi,T)>0$ such that
\begin{align*}
\norm[\Wom{2}{p}]{\varphi v(\cdot,t)}\leq C_3\qquad\text{for all }t\in(0,T).
\end{align*}
Due to the Sobolev embedding $W^{2,p}(\Omega)\hookrightarrow W^{1,q}(\Omega)$ with $q:=\frac{2p}{2-p}>2$ due to $p>1$, this implies \eqref{eq:varphi_v_w1q:est}.  
 %%% 
%  Elliptic regularity (\red{maybe Theorem~19.1 in Friedman?) gives
%  \begin{align*}
%    \norm[\Wom2{p}]{φv}\le C\left[\norm[\Lom{p}]{φu} + \norm[\Lom{p}]{v}\right]
%  \end{align*}
%  and thus
%  \begin{align*}
%    \norm[\Wom2{p}]{φv}\le C'
%  \end{align*}
%  by Lemma~\ref{lm:u_lp} and Lemma~\ref{lm:v_lq}.
%  Due to a Sobolev embedding theorem, this implies \eqref{eq:varphi_v_w1q:est} with some $C \gt 0$
%  for $q \defs \frac{2p}{2-p}$ if $p < 2$ and $q \defs 3$, say, if $p \ge 2$.
%  As $p > 1$, we have $q > 2$.
 %%%
\end{proof}

\begin{lem}\label{lm:local_bound_1}
  Let $x_0 \in \Ombar$ with $\mu(x_0) \gt 0$ and $T \in (0, \tmax] \cap (0, \infty)$.
  Then there exist an open neighbourhood $V$ of $x_0$ in $\Ombar$, $q > 2$ and $C > 0$ such that
  \begin{align}\label{eq:local_bound_1:est}
    \|\nabla v(\cdot, t)\|_{L^q(V)} < C
    \qquad \text{for all $t \in (0, T)$}.
  \end{align}
\end{lem}
\begin{proof}
  We set $\mu_0 \defs \frac{\mu(x_0)}{2} > 0$ as well as $U \defs \mu^{-1}((\mu_0, \infty)) \cap W$, where $W$ is the neighbourhood (open in $\Ombar$) of $x_0$ given by Lemma~\ref{lm:cutoff}.
  Moreover, we fix an open neighbourhood $V$ of $x_0$ in $\Ombar$ with $K \defs \bar V \subset U$
  and choose $p \in (1, \frac{1}{(1-\mu_0)_+})\cap(1,2)$ as well as $\eta \defs \frac{1}{2(p+1)}$.
  
  Lemma~\ref{lm:cutoff} then provides us with $\varphi \in C^2(\Ombar)$ and $C_\varphi \gt 0$ such that
  \eqref{eq:cutoff:varphi} and \eqref{eq:cutoff:grad_est} are fulfilled.
  In particular, Lemma~\ref{lm:varphi_v_w1q} is applicable so that letting $q:=\frac{2p}{2-p}>2$ we obtain $C=C(\mu_0,p,\varphi,T) > 0$ such that \eqref{eq:varphi_v_w1q:est} holds.
  Since $\varphi\equiv 1$ in $V$ and thus 
  \begin{align*}
        \|\nabla v(\cdot, t)\|_{L^q(V)}
    =   \|\nabla (\varphi v(\cdot, t))\|_{L^q(V)}
    \le \|\nabla (\varphi v(\cdot, t))\|_{L^q(\Omega)}
    \qquad \text{for all $t \in (0, T)$},
  \end{align*}
  this entails \eqref{eq:local_bound_1:est}.
\end{proof}

\section{\texorpdfstring{A localized $L^\infty$ bound for $u$}{A localized L-infinity bound for u}}\label{sec5}
In the case of $\varphi\equiv1$, the bounds on $\varphi^{\frac1p} u$ and $\nabla (\varphi v)$ in \eqref{eq:u_lp:statement} and \eqref{eq:varphi_v_w1q:est} can immediately enter a semigroup-based reasoning and provide further bounds for $u$ (see \cite[Lemma 3.2]{BellomoEtAlMathematicalTheoryKeller2015}). If, however, $\varphi$ is nonconstant and can vanish in some parts of the domain, 
\eqref{eq:u_lp:statement} and \eqref{eq:varphi_v_w1q:est} no longer control the precise combination terms to which the heat semigroup would need to be applied (cf.\ also \eqref{l_infty_bdd:est1} below).
% Unlike in the case where $\varphi \equiv 1$ (which is treated for instance in \cite[Lemma 3.2]{BellomoEtAlMathematicalTheoryKeller2015}),
% for $\varphi$ as in Lemma~\ref{lm:u_lp} and Lemma~\ref{lm:varphi_v_w1q}, it seems unclear how to infer further bounds for $\varphi u$ from \eqref{eq:u_lp:statement} and \eqref{eq:varphi_v_w1q:est}.
Instead, we will now first choose a \emph{new} cut-off function $\tilde\varphi$ with support in $V$ for the set $V$ given by Lemma~\ref{lm:local_bound_1},
so that we can control not only $\|\varphi v(\cdot, t)\|_{\sob1q}$ but even the norm $\|v(\cdot, t)\|_{W^{1, q}(\supp \tilde\varphi)}$ of $v$ itself, without spatial weight, for some $q > 2$ and $t \in (0, \tmax)$.

While restricting the domain of interest once, or even finitely many times, to smaller open sets is evidently fine,
applying this procedure infinitely often could be problematic, as the sets' size could shrink to zero. 
In particular, Moser-type iterations, as used frequently for quasilinear chemotaxis systems (see for instance the quite general Lemma~A.1 in \cite{TaoWinklerBoundednessQuasilinearParabolic2012}),
appear to be inadequate for our purposes.
Fortunately, for linear elliptic operators, one can obtain an $L^\infty$ bound in a single (additional) step by making use of semigroup arguments
(cf.\ the proof of \cite[Lemma 3.2]{BellomoEtAlMathematicalTheoryKeller2015}, for instance).
  
As a preparation for the localized $L^\infty$ bound for $u$, we first state the following elementary
\begin{lem}\label{lm:eta_trick}
  Let $T \in (0, \tmax] \cap (0, \infty)$, $\lambda' \in (1, \infty)$, $\eta' \in (0, \frac1{\lambda'}]$ and $\varphi \in \con0$.
  Then there exists $C > 0$ such that
  \begin{align*}
    \|\varphi^{1-\eta'} u(\cdot, t)\|_{\leb{\lambda'}} \le C \|\varphi u(\cdot, t)\|_{\leb \infty}^{1-\eta'}
    \qquad \text{for all $t \in (0, T)$}.
  \end{align*}
\end{lem}
\begin{proof}
  The statement follows immediately by writing $u = u^{1-\eta'} u^{\eta'}$ and applying Lemma~\ref{lm:u_l1},
  which provides a bound for $\sup_{t \in (0, T)} \io u^{\eta' \lambda'}(\cdot, t)$, due to $\eta'\lambda'\leq 1$.
\end{proof}

We directly make use of this lemma in a semigroup representation of the solution, combining it with bounds prepared in earlier sections.
\begin{lem}\label{lm:l_infty_bdd}
  Let $x_0 \in \Ombar$ with $\mu(x_0) > 0$ and $T \in (0, \tmax] \cap (0, \infty)$.
  Then there exists an open neighbourhood $V$ of $x_0$ in $\Ombar$ such that for any compact $K \subset V$,
  we can find $\varphi \in C^{2}(\Ombar)$ fulfilling \eqref{eq:cutoff:varphi} and $C \gt 0$ such that
  \begin{align}\label{eq:l_infty_bdd:est}
    \|\varphi u(\cdot, t)\|_{\leb\infty} \le C
    \qquad \text{for all $t \in (0, T)$}.
  \end{align}
\end{lem}
\begin{proof}
  First, Lemma~\ref{lm:local_bound_1} provides us with an open set $V \subset \Ombar$, $q > 2$ and $C_1 \gt 0$ such that $x_0 \in V$ and that \eqref{eq:local_bound_1:est} holds with $C$ replaced by $C_1$.
  Without loss of generality, we may assume $V \subset W$ for the set $W$ given by Lemma~\ref{lm:cutoff}.
  We now fix a compact $K \subset V$ and $\lambda \in (2, q)$, which is possible due to $q > 2$.

  Thus,
%  \begin{align}\label{eq:l_infty_bdd:eta}
%    \eta \defs \min\left\{\frac14, \frac{\lambda}{2}, \frac{q\lambda}{q-\lambda}\right\}
%  \end{align}
\begin{align}\label{eq:l_infty_bdd:eta}
    \eta \defs \min\left\{\frac14, \frac{1}{\lambda}, \frac{q-\lambda}{q\lambda}\right\}
  \end{align}
  is positive and finite.
  Moreover, by Lemma~\ref{lm:cutoff}, there exist $\varphi \in C^{2}(\Ombar)$ and $C_\varphi > 0$ satisfying \eqref{eq:cutoff:varphi} and \eqref{eq:cutoff:grad_est}. 
  Direct calculations show that
  \begin{align*}
          \Delta (\varphi u)
   % &=    \nabla \cdot (\varphi \nabla u + u \nabla \varphi)
   % =     \varphi \Delta u + \nabla u \cdot \nabla \varphi + \nabla \cdot (u \nabla \varphi)
    =     \varphi \Delta u - u \Delta \varphi + 2 \nabla \cdot (u \nabla \varphi)
%   \end{align*}
\quad \text{ and } \quad
%   \begin{align*}
          \nabla \cdot (\varphi u \nabla v)
    &=    \varphi \nabla \cdot (u \nabla v)
          + u \nabla \varphi \cdot \nabla v
  \end{align*}
  hold in $\Omega \times (0, T)$, whence
  we conclude from \eqref{system:u} that $\varphi u$ satisfies
  \begin{align*}
          (\varphi u)_t
    &=    \varphi \Delta u
          - \varphi \nabla \cdot (u \nabla v)
          + \kappa \varphi u
          - \mu \varphi u^2 \\
    &\le  \Delta(\varphi u) + u \Delta \varphi - 2 \nabla \cdot (u \nabla \varphi)
          - \nabla \cdot (\varphi u \nabla v) + u \nabla \varphi \cdot \nabla v
          + \kappa \varphi u
  \end{align*}
  in $\Omega \times (0, T)$ by the nonnegativity of $\mu$, $\varphi$ and $u$ throughout $\Omega\times(0,T)$.
  Since also
  \begin{align*}
      \partial_\nu (\varphi u)
    = u \partial_\nu \varphi + \varphi \partial_\nu u
    = 0
    \qquad \text{on $\partial \Omega \times (0, T)$}
  \end{align*}
  by \eqref{eq:cutoff:varphi} and \eqref{system:bc},
  we may apply well-known $L^p$-$L^q$ estimates (cf.\ \cite[Lemma~1.3]{WinklerAggregationVsGlobal2010}) to obtain that with some $C_2>0$ 
  \begin{align}\label{l_infty_bdd:est1}
    &\pe  \|\varphi u(\cdot, t)\|_{\leb \infty} \nn\\
    &\le  \|\ure^{t \Delta} (\varphi u_0)\|_{\leb \infty} \nn\\
    &\pe  + \int_0^t \left\| \ure^{(t-s) \Delta} \bigg(
              u(\cdot, s) \Delta \varphi + u(\cdot, s) \nabla \varphi \cdot \nabla v(\cdot, s) + \kappa \varphi u(\cdot, s)
            \bigg) \right\|_{\leb \infty} \mathrm{d}s\nn\\
    &\pe  + \int_0^t \left\| \ure^{(t-s) \Delta} \nabla \cdot \bigg(
              2u(\cdot, s) \nabla \varphi + \varphi u(\cdot, s) \nabla v(\cdot, s) 
            \bigg) \right\|_{\leb \infty} \mathrm{d}s \nn\\
    &\le  \|u_0\|_{\leb \infty} \nn\\
    &\pe  + C_2 \sup_{s \in (0, t)} \left\|
              u(\cdot, s) \Delta \varphi + u(\cdot, s) \nabla \varphi \cdot \nabla v(\cdot, s) + \kappa \varphi u(\cdot, s)
            \right\|_{\leb {\lambda}}
            \int_0^t \left(1 + (t-s)^{-\frac{2}{2\lambda}}\right) \,\mathrm{d}s\nn\\
    &\pe  + C_2 \sup_{s \in (0, t)} \left\|
              2u(\cdot, s) \nabla \varphi + \varphi u(\cdot, s) \nabla v(\cdot, s) 
            \right\|_{\leb{\lambda}}
            \int_0^t \left(1 + (t-s)^{-\frac12 - \frac{2}{2\lambda}}\right) \,\mathrm{d}s
  \end{align}
  holds for all $t \in (0, T)$, where we also used that $\varphi\leq 1$ in $\Omega$.
Herein, we make use of \eqref{eq:cutoff:grad_est} and $\varphi\leq 1$ in $\Ombar$ to estimate
  \begin{align}\label{l_infty_bdd:est2}
          \|u(\cdot, s) \Delta \varphi\|_{\leb{\lambda}}
    &\le  C_\varphi \|\varphi^{1-2\eta} u(\cdot, s)\|_{\leb{\lambda}}, \\
          \|\kappa \varphi u(\cdot, s)\|_{\leb{\lambda}}\label{l_infty_bdd:est3}
    &\le  \|\kappa\|_{\con0} \|\varphi^{1-\eta} u(\cdot, s)\|_{\leb{\lambda}},
  \intertext{and}\label{l_infty_bdd:est4}    
          \|u(\cdot, s) \nabla \varphi\|_{\leb{\lambda}}
    &\le  C_\varphi \|\varphi^{1-\eta} u(\cdot, s)\|_{\leb{\lambda}}
  \end{align}
  for all $s \in (0, T)$ and, similarly, this time also relying on $\varphi\equiv 0$ in $\Ombar\setminus V$ from \eqref{eq:cutoff:varphi}, Hölder's inequality and the choice of $C_1$ (cf.\ \eqref{eq:local_bound_1:est}), also
  \begin{align}\label{l_infty_bdd:est5}
   \|u(\cdot, s) \nabla \varphi \cdot \nabla v(\cdot, s)\|_{\leb{\lambda}}&\le \|u(\cdot, s) \nabla \varphi \cdot \nabla v(\cdot, s)\|_{\leb[V]{\lambda}}\nn\\
    &\le  C_\varphi \|\varphi^{1-\eta} u(\cdot, s)\|_{\leb[V]{\frac{q \lambda}{q-\lambda}}} \|\nabla v(\cdot, s)\|_{\leb[V]{q}}\nn\\
    &\le C_1 C_\varphi \|\varphi^{1-\eta} u(\cdot, s)\|_{\leb{\frac{q \lambda}{q-\lambda}}}\qquad\text{for all }s\in(0,T)
   \end{align}
   as well as
   \begin{align}\label{l_infty_bdd:est6}
          \|\varphi u(\cdot, s) \nabla v(\cdot, s)\|_{\leb \lambda} 
    &\le  \|\varphi^\eta\varphi^{1-\eta} u(\cdot, s)\|_{\leb[V]{\frac{q \lambda}{q-\lambda}}} \|\nabla v(\cdot, s)\|_{\leb[V]{q}}\nn\\
    &\le C_1 \|\varphi^{1-\eta} u(\cdot, s)\|_{\leb{\frac{q \lambda}{q-\lambda}}}\qquad\text{for all }s\in(0,T).
  \end{align}
  In light of \eqref{eq:l_infty_bdd:eta}, Lemma~\ref{lm:eta_trick} becomes applicable in all of the expressions above. In particular, we find from one application with $\lambda'=\lambda$ and $\eta'=\eta$ that there is $C_3>0$ such that 
\begin{align}\label{l_infty_bdd:est7}
\|\varphi^{1-\eta} u(\cdot, s)\|_{\leb{\lambda}}\leq C_3\|\varphi u(\cdot, s)\|_{\leb{\infty}}^{1-\eta}\qquad\text{for all }s\in(0,T),
\end{align}  
  and from another application with $\lambda'=\frac{q\lambda}{q-\lambda}\in(1,\infty)$ and $\eta'=\eta$ we also obtain $C_4>0$ satisfying
  \begin{align}\label{l_infty_bdd:est8}
  \|\varphi^{1-\eta} u(\cdot, s)\|_{\leb{\frac{q \lambda}{q-\lambda}}}\leq C_4\|\varphi u(\cdot, s)\|_{\leb{\infty}}^{1-\eta}\qquad\text{for all }s\in(0,T).
\end{align}    
Hence, we obtain from a combination of \eqref{l_infty_bdd:est1}--\eqref{l_infty_bdd:est8} that there is $C_5 \ge 1$ such that
  \begin{align*}
          \|\varphi u(\cdot, t)\|_{\leb \infty}
    &\le  C_5  \left(1 + \sup_{s \in (0, t)} \|\varphi u(\cdot, t)\|_{\leb \infty}^{1-\eta} \right) \\
    &\le  2^{\eta} C_5 \left(1 + \sup_{s \in (0, t)} \|\varphi u(\cdot, t)\|_{\leb \infty} \right)^{1-\eta}
    \qquad \text{for all $t \in (0, T)$}.
  \end{align*}
  In particular, $M \colon [0, T)\to \R, t \mapsto 1 + \sup_{s \in (0, t)} \|\varphi u(\cdot, s)\|_{\leb \infty}$, fulfils
  \begin{align*}
    M(t) \le 2^{1+\eta} C_5  M^{1-\eta}(t)
    \qquad \text{for all $t \in (0, T)$},
  \end{align*}
  which implies
  \begin{align*}
    M(t) \le (2^{1+\eta} C_5 )^\frac1\eta
    \qquad \text{for all $t \in (0, T)$},
  \end{align*}
  and thus the statement.
\end{proof}

\begin{lem}\label{lm:local_bound_2}
  Let $x_0 \in \Ombar$ with $\mu(x_0) > 0$ and $T \in (0, \tmax] \cap (0, \infty)$.
  Then there exist an open neighbourhood $U$ of $x_0$ in $\Ombar$ and $C \gt 0$ such that
  \begin{align*}
    \|u(\cdot, t)\|_{L^\infty(U)} \lt C 
    \qquad \text{for all $t \in (0, T)$}.
  \end{align*}
\end{lem}
\begin{proof}
  With $V$ as given by Lemma~\ref{lm:l_infty_bdd}, we choose an open neighbourhood $U$ of $x_0$ with $K \defs \bar U \subset V$.
  We are then able to apply Lemma~\ref{lm:l_infty_bdd} to obtain $\varphi \in C^\infty(\Ombar)$ and $C \gt 0$ such that \eqref{eq:cutoff:varphi} and \eqref{eq:l_infty_bdd:est} hold.
  This already entails the statement since $\varphi u = u$ in $U \times (0, T)$.
\end{proof}

\section{Conclusion: Proofs of Theorem~\ref{th:main} and Theorem~\ref{th:blowup_points}}
\begin{proof}[Proof of Theorem~\ref{th:main}]
  See Lemma~\ref{lm:local_bound_2}.
\end{proof}

\begin{proof}[Proof of Theorem~\ref{th:blowup_points}]
  Suppose $(u, v)$ is a solution of \eqref{system} blowing up at $\tmax < \infty$ and let $x_0 \in \Ombar$ be such that $\mu(x) > 0$.
  Theorem~\ref{th:main} then directly asserts that $x_0$ does not belong to the blow-up set.
\end{proof}

\section*{Acknowledgments}
T.B. acknowledges support of the \emph{Deutsche Forschungsgemeinschaft} \ in the context of the project \emph{Emergence of structures and advantages in
cross-diffusion systems} \  (No.\ 411007140, GZ: WI 3707/5-1).

%\bibliography{bib}

\end{document}